\documentclass[10pt,final,twocolumn]{IEEEtran}
\long\def\comment#1{}

\usepackage{textcomp}
\usepackage{cite}
\usepackage{graphicx}
\usepackage{subfigure}
\usepackage{psfrag}
\usepackage{url}
\usepackage{stfloats}
\usepackage{amsmath}
\usepackage{amssymb,amsfonts}
\usepackage{amsthm}
\usepackage{float}
\usepackage[usenames]{color}
\usepackage{algorithm,algorithmic}

\newtheorem{assumption}{Assumption}
\newtheorem{remark}{Remark}

\newtheorem{lemma}{Lemma}
\newtheorem{theorem}{Theorem}
\newtheorem{example}{Example}

\begin{document}

\setlength{\arraycolsep}{0.3em}

\title{Online Abstract Dynamic Programming with Contractive Models
\thanks{}}

\author{Xiuxian Li and Lihua Xie
\thanks{X. Li is with Department of Control Science and Engineering, College of Electronics and Information Engineering, Institute for Advanced Study, and Shanghai Research Institute for Intelligent Autonomous Systems, Tongji University, Shanghai, China (e-mail: xli@tongji.edu.cn).}
\thanks{L. Xie is with School of Electrical and Electronic Engineering, Nanyang Technological University, 50 Nanyang Avenue, Singapore 639798 (e-mail: elhxie@ntu.edu.sg).}
}

\maketitle

\setcounter{equation}{0}
\setcounter{figure}{0}
\setcounter{table}{0}

\begin{abstract}
This paper addresses the abstract dynamic programming (DP) in the online scenario, where the abstract DP mapping is time-varying, instead of static. In this case, optimal costs and policies at different time instants are not the same in general, and the problem amounts to tracking time-varying optimal costs and policies, which is of interest to many practical problems. It is thus necessary to analyze the performance of classical value iteration (VI) and policy iteration (PI) algorithms in the online case. In doing so, this paper develops and provides the theoretical analysis for several online algorithms, including approximate online VI, online PI, approximate online PI, online optimistic PI, approximate online optimistic PI, and asynchronous online PI and VI algorithms. It is proved that the tracking error bounds for all algorithms critically depend upon the largest difference between any two consecutive abstract mappings. Meanwhile, examples are presented to illustrate the theoretical results.
\end{abstract}


\section{Introduction}\label{s1}

Dynamic programming (DP) is a powerful tool in handling total cost sequential decision problems, which has been extensively investigated up to now and can find lots of applications in optimal control, Markovian decision problems (MDPs), stochastic shortest path problems (SSP), zero-sum dynamic game, and reinforcement learning, and so on \cite{bertsekas2018abstract,bertsekas2018proper,bertsekas2015value,yang2017hamiltonian,liu2013policy,wei2015value,heydari2014revisiting,song2014adaptive,seiffertt2008hamilton,
chang2006policy,ni2015model,bucsoniu2011approximate}. In this paper, the focus is on abstract DP, which provides a unified analysis for DP models by abstracting their substantial structures.

In general, the models for abstract DP are classified into three types. The first is the contractive models, where there exists an abstract mapping that is a contraction over a space consisting of bounded functions defined on the state space, which is first introduced in \cite{denardo1967contraction}. These models have well-behaved analytical and computational properties. The second is the semicontractive models, introduced in \cite{bertsekas2018abstract}, and in this case, the abstract mapping is no longer a contraction over the whole bounded function space. However, in this model, some policies possess a contraction-like property while others do not, and these models can have a good enough theory nearly as in the contractive models when certain conditions hold. The third is the noncontractive models \cite{bertsekas1975monotone,bertsekas1977monotone}, in which the abstract mapping is monotone, instead of contractive. It is known that pathologies emerge in the noncontractive models, leading to that it is difficult to seek effective solutions \cite{bertsekas2017regular}.

There are mainly two fundamental algorithms in abstract DP, i.e., value iteration (VI) and policy iteration (PI), based on which various algorithms have been developed, including approximate VI and PI in finite-state discounted MDP \cite{scherrer2012on}, optimistic PI (or modified PI) \cite{puterman1994markovian}, approximate optimistic PI \cite{canbolat2013approximate}, $\lambda$-PI method \cite{bertsekas1996neuro}, approximate $\lambda$-PI method \cite{thiery2010least}, asynchronous VI \cite{bertsekas1982distributed}, and asynchronous PI \cite{williams1993analysis}. The core of VI and PI is the so-called Bellman's equation, and the key point is to find a fixed point of the corresponding mapping to Bellman's equation.

To date, although there exist numerous works on abstract DP problems as discussed above, most of them are devoted to the case of stationary abstract DP mappings. Nevertheless, in practical problems one often encounters the scenarios where the abstract DP mapping is time-varying itself or caused by the environment's uncertainties, that is, the cost function is time-varying and one usually does not have enough time to perform offline calculation for completely solving the problem at each time step before it goes forward to the next time step. For instance, when tracking a moving target for an unmanned aerial vehicle (UAV), the cost for penalizing the distance between this vehicle and the target is apparently time-varying. To meet the needs of practical applications, such as in reinforcement learning, researchers in optimization, machine learning, and control communities, etc. have put their great effort on online optimization/learning, where the cost functions are time-varying and gradually revealed to the decision maker, that is, the decision maker only knows the information on cost functions at hand until now, without aware of future information. Of pertinent literature along this line are \cite{shahrampour2018distributed,dixit2019online,bliek2016online,li2018distributedon,li2021distri,yi2019distributed,yi2021distri,bernstein2018asynchronous}, to just name a few.

Motivated by the above discussions, this paper aims to study the abstract DP problems with time-varying abstract DP mappings, called {\em online (or running) abstract DP} problems in this paper. To the best of our knowledge, it is the first time to consider the online scenario for abstract DP problems. Of closely relevant work is \cite{bernstein2018asynchronous}, which investigated the fixed point seeking problem for a time-varying sequence of contractive mappings or operators. However, the results in \cite{bernstein2018asynchronous} is unavailable in the context of abstract DP since policy iteration in abstract DP is more complicated than that in \cite{bernstein2018asynchronous}. The contributions of this paper are to develop and analyze online algorithms for online abstract DP, including approximate online VI, online PI, approximate online PI, online optimistic PI, approximate online optimistic PI, and asynchronous online PI and VI algorithms. It is shown that all error bounds for optimal cost tracking are closely related to the differences between consecutive mappings $H_k$ and $H_{k+1}$ for $k\geq 0$.

This paper is organized as follows. Section \ref{s2} formulates the problem, and online PI and VI algorithms in the synchronous case are discussed in Sections \ref{s3} and \ref{s4}, respectively. The asynchronous online PI and VI algorithms are given in Section \ref{s5}, following examples in Section \ref{s6}. Finally, the conclusion is drawn in Section \ref{s7}.

\section{Problem Formulation}\label{s2}

Let $\mathbb{R}$ and $\mathbb{N}$ be the sets of real numbers and nonnegative integers, respectively. Denote by $X$ and $U$ two sets, which can be roughly viewed as the sets of ``states'' and ``controls'', respectively. Given a state $x\in X$, let $U(x)\subset U$ be a subset of $U$, denoting feasible controls at state $x$. Let $\mathcal{M}:=\{\mu: X\to U|~\mu(x)\in U(x),~\forall~x\in X\}$, representing a collection of functions. Similar to DP, a sequence $\{\mu_k\}_{k=0}^\infty$ with $\mu_k\in\mathcal{M}$ for all $k\in\mathbb{N}$ is called a {\em nonstationary policy}, and if all $\mu_k$'s are identical, that is, $\mu_k=\mu$ for some $\mu\in\mathcal{M}$ for all $k\in\mathbb{N}$, then it is called a {\em stationary policy}. To simplify the notation, any single $\mu\in\mathcal{M}$ is also referred to as a {\em policy} when $\{\mu\}$ is a stationary policy.

Denote by $\mathcal{R}(X)$ a set of real-valued functions $J: X\to \mathbb{R}$. In online abstract DP, consider a family of time-varying mappings $H_k: X\times U\times \mathcal{R}(X)\to \mathbb{R}$, where $k\in\mathbb{N}$ is interpreted as time index. The mappings $\{H_k\}_{k=0}^\infty$ are only gradually revealed: at each time $k\in\mathbb{N}$, we only know the mappings before time $k$, but without awareness of future information on $H_l$ for $l\geq k$. Given a time $k\in\mathbb{N}$ and a policy $\mu\in\mathcal{M}$, let us consider the mapping $T_{k,\mu}: \mathcal{R}(X)\to\mathcal{R}(X)$ defined as
\begin{align}
(T_{k,\mu} J)(x)=H_k(x,\mu(x),J),~~\forall~x\in X,~J\in\mathcal{R}(X)           \label{1}
\end{align}
and also consider a mapping $T_k: \mathcal{R}(X)\to\mathcal{R}(X)$ defined as
\begin{align}
(T_k J)(x)&=\inf_{u\in U(x)}H_k(x,u,J)           \nonumber\\
&=\inf_{\mu\in\mathcal{M}}(T_{k,\mu})(x),~~~\forall~x\in X,~J\in\mathcal{R}(X).        \label{2}
\end{align}

The objective of online abstract DP is to find a function $J_k^*\in\mathcal{R}(X)$ at each time $k$ such that
\begin{align}
J_k^*(x)=\inf_{u\in U(x)}H_k(x,u,J_k^*),           \label{3}
\end{align}
i.e., seeking a fixed point of $T_k$ at each time step $k\in\mathbb{N}$, which is typically called {\em Bellman's equation}. Meanwhile, it is desirable to obtain a policy $\mu_k^*\in\mathcal{M}$ such that $T_{k,\mu_k^*}J_k^*=T_kJ_k^*$. That is, $\mu_k^*$ is an {\em optimal policy} corresponding to the {\em optimal cost} $J_k^*$.

The following is an example for illustrating the above problem.
\begin{example}[Online Optimal Control]\label{e1}
Consider a deterministic discrete-time online optimal control problem, where a nonlinear control system is given as
\begin{align}
x_{k+1}=f(x_k,u_k),~~~k\in\mathbb{N}              \label{4}
\end{align}
with $x_k\in X$ and $u_k\in U$ being the state and control of the system, respectively.

At each time slot $k\in\mathbb{N}$, there is an objective or cost function $g_k(x,u)$, and the aim is to minimize the total cost incurred by a policy $\pi_k=\{\mu_k,\mu_{k+1},\ldots\}$ over an infinite number of stages with the initial state $x_k$, i.e.,
\begin{align}
\text{minimize}~~J_{\pi_k}(x_k):=\sum_{m=0}^\infty \alpha^m g_k(x_{k+m},\mu_{k+m}),         \label{5}
\end{align}
where $\alpha\in(0,1]$ is a discounted factor.

The optimal cost function is defined by
\begin{align}
J_k^*(x)=\inf_{\pi_k\in\Pi_k} J_{\pi_k}(x),~~~\forall~x\in X,           \label{6}
\end{align}
where
\begin{align}
\Pi_k:=\{\{\mu_k,\mu_{k+1},\ldots\}|~\mu_m\in\mathcal{M},~\forall~m\geq k\}.            \label{7}
\end{align}

For arbitrary policy $\pi_k=\{\mu_k,\mu_{k+1},\ldots\}$ and writing $\pi_{k+1}=\{\mu_{k+1},\mu_{k+2},\ldots\}$, one can easily rewrite $J_{\pi_k}(x)$ as
\begin{align}
J_{\pi_k}(x)=g_k(x,\mu_k)+\alpha J_{\pi_{k+1}}(f(x,\mu_k)),~~\forall~x\in X             \label{8}
\end{align}
which leads to that
\begin{align}
J_k^*(x)&=\inf_{\pi_k=\{\mu_k,\pi_{k+1}\}\in\Pi_k}\Big\{g_k(x,\mu_k)+\alpha J_{\pi_{k+1}}(f(x,\mu_k))\Big\}      \nonumber\\
&=\inf_{\mu_k\in\mathcal{M}}\Big\{g_k(x,\mu_k)+\alpha\inf_{\pi_{k+1}\in\Pi_{k+1}} J_{\pi_{k+1}}(f(x,\mu_k))\Big\}      \nonumber\\
&=\inf_{\mu_k\in\mathcal{M}}\Big\{g_k(x,\mu_k)+\alpha J_k^*(f(x,\mu_k))\Big\}.              \label{9}
\end{align}
Once defining $H_k(x,u,J)=g_k(x,u)+\alpha J(f(x,u))$, through the above equation, it is easy to see that
\begin{align}
J_k^*(x)=\inf_{u\in U(x)} H_k(x,u,J_k^*),~~\forall x\in X,            \label{10}
\end{align}
which is exactly consistent with (\ref{3}). As a result, this online optimal control problem can be viewed as an instance of online abstract DP.
\end{example}

More examples for abstract DP can be found in \cite{bertsekas2018abstract}, including stochastic Markovian decision problems, finite-state discounted Markovian decision problems, discounted semi-Markov problems, discounted zero-sum dynamic games, minimax problems, and stochastic shortest path problems, etc. It should be noted that online abstract DP will reduce to abstract DP when the mapping $H_k$ is time-invariant.

To proceed, it is necessary to introduce a new space $\mathcal{B}(X)$, composed of functions $J$ on $X$ such that $J(x)/\nu(x)$ is bounded for all $x\in X$, where $\nu: X\to\mathbb{R}$ is a function with $\nu(x)>0$ for all $x\in X$. On the space $\mathcal{B}(X)$, a {\em weighted sup-norm} is defined as
\begin{align}
\|J\|=\sup_{x\in X}\frac{|J(x)|}{\nu(x)}.                  \label{11}
\end{align}
It has been shown in Appendix B of \cite{bertsekas2018abstract} that $\mathcal{B}(X)$ is a complete normed space with respect to the weighted sup-norm.

At this moment, two important assumptions are listed below.

\begin{assumption}[Monotonicity]\label{a1}
For all $k\in\mathbb{N}$ and any $J_1,J_2\in\mathcal{R}(X)$, if $J_1\leq J_2$, then
\begin{align}
H_k(x,u,J_1)\leq H_k(x,u,J_2),~~~\forall x\in X,~u\in U(x).             \label{12}
\end{align}
\end{assumption}

\begin{assumption}[Uniform Contraction]\label{a2}
For all $k\in\mathbb{N}$, $J\in\mathcal{B}(X)$, and $\mu\in\mathcal{M}$, there holds $T_{k,\mu}J,T_k J\in\mathcal{B}(X)$. Moreover, there exists $\alpha_k\in(0,1)$ such that for all $k\in\mathbb{N}$ and $\mu\in\mathcal{M}$
\begin{align}
\|T_{k,\mu}J_1-T_{k,\mu}J_2\|\leq \alpha_k\|J_1-J_2\|,~\forall J_1,J_2\in\mathcal{B}(X)         \label{13}
\end{align}
and $\alpha:=\max_{k\in\mathbb{N}}\alpha_k\in(0,1)$.
\end{assumption}

It is noteworthy that the monotonicity assumption holds for almost all relevant DP mappings, and the weighted sup-norm contraction assumption is satisfied for a multitude of important DP models, such as discounted finite-state MDP, and undiscounted finite-state SSP with all policies being proper. More discussions can be found in \cite{bertsekas2018abstract}.

To conclude this section, the following lemma is conducive to the ensuing analysis, which can be found in \cite{bernstein2018asynchronous}.
\begin{lemma}\label{l0}
For a positive sequence $\{a_k\}$, if there exist $K<1$, $b<1$, and $0<\tau<1$ such that for all $k>K$
\begin{align}
a_k\leq b+\tau a_{k-\delta_k},         \nonumber
\end{align}
for some $\delta_k\in\{1,\ldots,K\}$, then, $\mathop{\lim\sup}_{k\to\infty}a_k\leq \frac{b}{1-\tau}$.
\end{lemma}

\section{Synchronous Online Value Iteration}\label{s3}

This section is devoted to online VI algorithms' development and analysis in the synchronous setting.

As seen from (\ref{3}), the goal is to find the fixed point of $T_k$ at each time slot $k\in\mathbb{N}$. To this end, an {\em approximate VI} is proposed as
\begin{align}
J_{k+1}&=\tilde{T}_k^{m_k} J_k           \label{14}
\end{align}
with any initial condition $J_0\in\mathcal{B}(X)$, where $\tilde{T}_k J_k$ stands for an approximation of $T_k J_k$, satisfying
\begin{align}
\|\tilde{T}_k^{m_k}J-T_k^{m_k}J\|\leq e_k,~~\forall J\in\mathcal{B}(X)          \label{15}
\end{align}
with $e:=\max_{k\in\mathbb{N}}e_k<\infty$, and $m_k\geq 1$ is an integer, representing the computational power at time step $k\in\mathbb{N}$. For this online problem, it is of necessity to impose a condition on the switching rate of consecutive optimal costs, that is, there exists a constant $\rho_k\geq 0$ for each $k\in\mathbb{N}$ such that
\begin{align}
\|J_k^*-J_{k+1}^*\|\leq \rho_k,            \label{16}
\end{align}
and $\rho:=\max_{k\in\mathbb{N}}\rho_k<\infty$.

It is now ready to present the tacking error bound for the approximate VI (\ref{14}).
\begin{theorem}\label{t1}
Under Assumption \ref{a2}, there holds for $J_k$ generated by approximate VI (\ref{14}) that
\begin{align}
\|J_k-J_k^*\|\leq \alpha^{\sum_{s=0}^{k-1}m_s}\|J_0-J_0^*\|+\frac{\rho+e}{1-\alpha^{m_d}},         \label{17}
\end{align}
where $m_d:=\min_{k\in\mathbb{N}}m_k \geq 1$.
\end{theorem}

\begin{proof}
In view of (\ref{14}), it can be obtained that
\begin{align}
\|J_{k+1}-J_{k+1}^*\|&=\|\tilde{T}_k^{m_k}J_k-J_{k+1}^*\|          \nonumber\\
&\leq \|\tilde{T}_k^{m_k}J_k-T_k^{m_k}J_k\|+\|T_k^{m_k}J_k-J_k^*\|         \nonumber\\
&\hspace{0.4cm}+\|J_k^*-J_{k+1}^*\|          \nonumber\\
&\leq \alpha_k^{m_k}\|J_k-J_k^*\|+e_k+\rho_k,                \nonumber
\end{align}
where the second inequality has employed Assumption \ref{a2} and (\ref{15})-(\ref{16}). By recursion, one has that
\begin{align}
\|J_{k+1}-J_{k+1}^*\|&\leq \prod_{s=0}^k\alpha_s^{m_s}\|J_0-J_0^*\|+\sum_{s=0}^k \alpha_{k:s}(\rho_s+e_s)          \nonumber\\
&\hspace{-0.8cm}\leq \alpha^{\sum_{s=0}^k m_s}\|J_0-J_0^*\|+(\rho+e)\sum_{s=0}^k\alpha^{\sum_{r=s+1}^k m_r}          \nonumber\\
&\hspace{-0.8cm}\leq \alpha^{\sum_{s=0}^k m_s}\|J_0-J_0^*\|+\frac{\rho+e}{1-\alpha^{m_d}},                \nonumber
\end{align}
where $\alpha_{k:s}:=\prod_{r=s+1}^k\alpha_r^{m_r}$ when $t=0,1,\ldots,k-1$, and $\alpha_{k:s}:=1$ when $s=k$. This ends the proof.
\end{proof}

\begin{remark}\label{r1}
It is worth mentioning that a similar online algorithm for finding fixed points of time-varying mappings is addressed in \cite{bernstein2018asynchronous}, which is a special case of Theorem \ref{t1} with $m_k=1$ for all $k\in\mathbb{N}$. It can be observed from (\ref{17}) that $J_k$ will approach to $J_k^*$ with an error bound $(\rho+e)/(1-\alpha^{m_d})$ at an exponential rate as $k$ tends to infinity.
\end{remark}

\section{Synchronous Online Policy Iteration}\label{s4}

This section is concerned with the online PI algorithms in the synchronous setup, including exact/approximate online PI and optimistic PI algorithms.

\subsection{Online Policy Iteration}

First, let us consider the exact online PI for solving online abstract DP, for which, given the current policy $\mu_k$ with an initial policy $\mu_0$, the policy update $\mu_{k+1}$ at time step $k+1$ is given as
\begin{subequations}\label{18}
\begin{align}
J_{k,\mu_k}&=T_{k,\mu_k}J_{k,\mu_k},\text{\em (Online policy evaluation)}                  \label{18a}\\
\hspace{-0.19cm}T_{k,\mu_{k+1}}J_{k,\mu_k}&=T_k J_{k,\mu_k},\text{\em (Online policy improvement).}         \label{18b}
\end{align}
\end{subequations}
It is assumed that one can attain the minimum of $H_k(x,u,J_{k,\mu_k})$ over $u\in U(x)$ for all $x\in X$, such that the update $\mu_{k+1}$ at online policy improvement is well defined, and this assumption is always exploited for PI algorithms in this paper. The purpose of online policy evaluation (\ref{18a}) is to calculate $J_{k,\mu_k}$, i.e., to find the fixed point of $T_{k,\mu_k}$, and (\ref{18b}) is leveraged to obtain $\mu_{k+1}$.

To move forward, it is imperative to introduce the following bounds for the online abstract DP:
\begin{align}
\|J_{k,\mu}-J_{k+1,\mu}\|&\leq \gamma_{1,k},~~\forall \mu\in U       \nonumber\\
\|J_k^*-J_{k+1}^*\|&\leq \gamma_{2,k},~~\forall k\in\mathbb{N}       \label{19}
\end{align}
where $J_{k,\mu}$ is the fixed point of $T_{k,\mu}$ for any $k\in\mathbb{N}$ and $\mu\in \mathcal{M}$, the first inequality indicates to what extent $H_k$ is different from $H_{k+1}$ in the case of the same input, and the second one connotes the switching bound on consecutive optimal costs.

With the above preparations, the main result on online VI (\ref{18}) is given as follows.

\begin{theorem}\label{t2}
Under Assumptions \ref{a1} and \ref{a2}, there holds for online VI (\ref{18}) that for all $k\in\mathbb{N}$
\begin{align}
\|J_{k,\mu_k}-J_k^*\|\leq \alpha^k\|J_{0,\mu_0}-J_0^*\|+\frac{\gamma_1+\gamma_2}{1-\alpha},          \label{20}
\end{align}
where $\gamma_l:=\max_{k\in\mathbb{N}}\gamma_{l,k}$ for $l=1,2$.
\end{theorem}

\begin{proof}
Invoking (\ref{18}) and the definition of $T_k$, it can be concluded that
\begin{align}
T_{k,\mu_{k+1}}J_{k,\mu_k}=T_kJ_{k,\mu_k}\leq T_{k,\mu_k}J_{k,\mu_k}=J_{k,\mu_k},           \nonumber
\end{align}
which, together with Assumption \ref{a1}, follows that
\begin{align}
T_{k,\mu_{k+1}}^2J_{k,\mu_k}\leq T_{k,\mu_{k+1}}J_{k,\mu_k}=T_kJ_{k,\mu_k}\leq J_{k,\mu_k}.           \nonumber
\end{align}
Performing the above operation iteratively, one can obtain that
\begin{align}
T_{k,\mu_{k+1}}^mJ_{k,\mu_k}\leq T_{k}J_{k,\mu_k},~~\forall m\geq 1.           \nonumber
\end{align}
By letting $m\to\infty$, it results in
\begin{align}
J_{k,\mu_{k+1}}\leq T_{k}J_{k,\mu_k},           \nonumber
\end{align}
which yields by Assumption \ref{a2} that for all $x\in X$
\begin{align}
J_{k,\mu_{k+1}}(x)-J_k^*(x)&\leq T_{k}J_{k,\mu_k}(x)-J_k^*(x)         \nonumber\\
&\leq\alpha_k\|J_{k,\mu_k}-J_k^*\|\nu(x).           \label{pf1}
\end{align}

It is known that $J_k^*(x)=\inf_{\mu\in\mathcal{M}}J_{k,\mu}(x)$ for all $x\in X$ and $k\in\mathbb{N}$ by Proposition 2.1.2 in \cite{bertsekas2018abstract}, and $\alpha_k\leq \alpha$. Therefore, one has by (\ref{pf1}) that $J_{k,\mu_{k+1}}(x)\geq J_k^*(x)$ and
\begin{align}
\|J_{k,\mu_{k+1}}-J_k^*\|\leq \alpha \|J_{k,\mu_k}-J_k^*\|,           \nonumber
\end{align}
which in combination with (\ref{19}) leads to that
\begin{align}
\|J_{k+1,\mu_{k+1}}-J_{k+1}^*\|&\leq \|J_{k,\mu_{k+1}}-J_k^*\|+\|J_k^*-J_{k+1}^*\|         \nonumber\\
&\hspace{0.4cm}+\|J_{k+1,\mu_{k+1}}-J_{k,\mu_{k+1}}\|           \nonumber\\
&\hspace{-0.3cm}\leq\alpha\|J_{k,\mu_k}-J_k^*\|+\|J_k^*-J_{k+1}^*\|+\gamma_{1,k},            \nonumber
\end{align}
further implying (\ref{20}) by recursive iterations. This completes the proof.
\end{proof}

\begin{remark}\label{r2}
Note that unlike the case where $H_k$'s are time-invariant, it generally cannot guarantee the convergence of $\{\mu_k\}_{k=0}^\infty$ generated by online VI (\ref{18}) under arbitrary compactness
and continuity conditions, since there exists an error term $(\gamma_1+\gamma_2)/(1-\alpha)$ in the online case.
\end{remark}

\subsection{Approximate Online Policy Iteration}

In this subsection, let us consider the online policy iteration through approximations, called {\em approximate online policy iteration}, which generates a sequence of approximate cost functions $\{J_k\}$ and policies $\{\mu_k\}$ satisfying that for all $k\in\mathbb{N}$
\begin{align}
\|J_k-J_{k,\mu_k}\|\leq \delta_{1,k},~~\|T_{k,\mu_{k+1}}J_k-T_kJ_k\|\leq \epsilon_{1,k},          \label{21}
\end{align}
where $\delta_{1,k},\epsilon_{1,k}\geq 0$ are some constants. Then the following result can be obtained.

\begin{theorem}\label{t3}
Under Assumptions \ref{a1} and \ref{a2}, the sequences $\{\mu_k\}$ generated by approximate online PI (\ref{21}) satisfy
\begin{align}
\|J_{k,\mu_k}-J_k^*\|\leq \alpha^k\|J_{0,\mu_0}-J_0^*\|+\frac{r_1}{1-\alpha},             \label{22}
\end{align}
where $r_1:=\gamma_1+\gamma_2+(\epsilon_1+2\alpha\delta_1)/(1-\alpha)$ with $\epsilon_1:=\max_{k\in\mathbb{N}}\epsilon_{1,k}$, $\delta_1:=\max_{k\in\mathbb{N}}\delta_{1,k}$, and $\gamma_1,\gamma_2$ are defined in Theorem \ref{t2}.
\end{theorem}

\begin{proof}
For each $k\in\mathbb{N}$, in view of Proposition 2.4.4 in \cite{bertsekas2018abstract}, one can obtain that
\begin{align}
\|J_{k,\mu_{k+1}}-J_k^*\|\leq \alpha_k\|J_{k,\mu_k}-J_k^*\|+\frac{\epsilon_{1,k}+2\alpha_k\delta_{1,k}}{1-\alpha_k},      \nonumber
\end{align}
which implies that
\begin{align}
\|J_{k+1,\mu_{k+1}}-J_{k+1}^*\|&\leq \|J_{k,\mu_{k+1}}-J_k^*\|+\|J_k^*-J_{k+1}^*\|           \nonumber\\
&\hspace{0.4cm}+\|J_{k+1,\mu_{k+1}}-J_{k,\mu_{k+1}}\|                      \nonumber\\
&\leq \alpha_k\|J_{k,\mu_k}-J_k^*\|+\gamma_{1,k}+\gamma_{2,k}            \nonumber\\
&\hspace{0.4cm}+\frac{\epsilon_{1,k}+2\alpha_k\delta_{1,k}}{1-\alpha_k}      \nonumber\\
&\leq \alpha\|J_{k,\mu_k}-J_k^*\|+r_1,            \nonumber
\end{align}
where we have used (\ref{19}) in the second inequality and the facts $\alpha_k\leq\alpha,\gamma_{l,k}\leq\gamma_l$ for $k\in\mathbb{N},l=1,2$ in the last inequality. By recursively iterating the above inequality, the conclusion (\ref{22}) can be asserted.
\end{proof}

\subsection{Online Optimistic Policy Iteration}

In online PI, the online policy evaluation (\ref{18a}) requires to exactly resolve the fixed point of $T_{k,\mu_k}$, which is usually computationally prohibitive. To alleviate the computational burden, another algorithm, called {\em online optimistic PI (or online ``modified'' PI)}, aims to approximately solve the fixed point of $T_{k,\mu_k}$, delineated as for a given initial cost function $J_0\in\mathcal{B}(X)$
\begin{align}
T_{k,\mu_k}J_k=T_kJ_k,~~J_{k+1}=T_{k,\mu_k}^{m_k}J_k,            \label{23}
\end{align}
producing a sequence of $\{\mu_k\}$ and $\{J_k\}$, where $m_k\geq 1$ is an integer for iterating the mapping $T_{k,\mu_k}$ totally $m_k$ times dependent on the computation power at time step $k\in\mathbb{N}$. To analyze (\ref{23}), a metric to measure the consecutive difference between $T_k$ and $T_{k+1}$ is postulated as
\begin{align}
\|(T_k-T_{k+1})J\|\leq \eta_{1,k},~~\forall J\in\mathcal{B}(X)       \label{24}
\end{align}
for some constant $\eta_{1,k}\geq 0$ and for all $k\in\mathbb{N}$.

At this stage, it is helpful to introduce a preliminary result on the boundedness of $J_{k+1}$, which is an extension of Lemma 2.5.3 in \cite{bertsekas2018abstract} to the online case considered in this paper.

\begin{lemma}\label{l1}
Under Assumptions \ref{a1} and \ref{a2}, if $J_0\geq T_{0}J_0-c\nu$ for some $c\geq 0$, then for all $k\in\mathbb{N}$
\begin{align}
T_kJ_k+\frac{\alpha_k}{1-\alpha_k}\lambda_k(c)\nu &\geq J_{k+1}          \nonumber\\
&\geq T_{k+1}J_{k+1}-\lambda_{k+1}(c)\nu,         \label{25}
\end{align}
where $\lambda_k(c)$ is defined by
\begin{align}
\lambda_k(c)=\left\{
                \begin{array}{ll}
                  c, & if~k=0 \\
                  \sum_{s=0}^{k-1}\eta_{1,s}\prod_{l=s+1}^{k-1}\alpha_l^{m_l}+c\prod_{l=0}^{k-1}\alpha_l^{m_l}, & if~k\geq 1
                \end{array}
              \right.        \nonumber
\end{align}
with the convention $\prod_{l=s+1}^{k-1}\alpha_l^{m_l}=1$ when $s=k-1$.
\end{lemma}
\begin{proof}
Because of $J_0\geq T_0 J_0-c\nu$, in view of Lemma 2.5.2 in \cite{bertsekas2018abstract} by letting $T=T_0,J=J_0,k=m_0$, and $\mu=\mu_0$, one has
\begin{align}
T_0J_0\geq T_{\mu_0}^{m_0}J_0-\frac{\alpha_0 c\nu}{1-\alpha_0}=J_1-\frac{\alpha_0\lambda_0(c)\nu}{1-\alpha_0},        \nonumber
\end{align}
and
\begin{align}
J_1&=T_{\mu_0}^{m_0}J_0\geq T_0(T_{\mu_0}^{m_0}J_0)-\alpha_0^{m_0}c\nu=T_0J_1-\alpha_0^{m_0}c\nu         \nonumber\\
&=T_1J_1+(T_0-T_1)J_1-\alpha_0^{m_0}c\nu                 \nonumber\\
&\geq T_1J_1-\eta_{1,0} \nu-\alpha_0^{m_0}c\nu           \nonumber\\
&=T_1J_1-\lambda_1(c)\nu,                       \nonumber
\end{align}
where (\ref{24}) has been employed in the second inequality. Therefore, (\ref{25}) holds when $k=0$.

By induction, it is assumed that (\ref{25}) holds for $k\geq 1$, and then one has $J_k\geq T_kJ_k-\lambda_k(c)\nu$, which in conjunction with Lemma 2.5.2 in \cite{bertsekas2018abstract} with $T=T_k,J=J_k,k=m_k$ and $\mu=\mu_k$ yields that
\begin{align}
T_kJ_k\geq T_{\mu_k}^{m_k}J_k-\frac{\alpha_k\lambda_k(c)\nu}{1-\alpha_k}=J_{k+1}-\frac{\alpha_k\lambda_k(c)\nu}{1-\alpha_k},       \nonumber
\end{align}
and
\begin{align}
J_{k+1}&=T_{\mu_k}^{m_k}J_k\geq T_k(T_{\mu_k}^{m_k}J_k)-\alpha_k^{m_k}\lambda_k(c)\nu          \nonumber\\
&=T_kJ_{k+1}-\alpha_k^{m_k}\lambda_k(c)\nu               \nonumber\\
&=T_{k+1}J_{k+1}+(T_k-T_{k+1})J_{k+1}-\alpha_k^{m_k}\lambda_k(c)\nu           \nonumber\\
&\geq T_{k+1}J_{k+1}-\eta_{1,k}\nu-\alpha_k^{m_k}\lambda_k(c)\nu               \nonumber\\
&=T_{k+1}J_{k+1}-\lambda_{k+1}(c)\nu,                            \nonumber
\end{align}
where the second inequality has leveraged (\ref{24}). This ends the proof.
\end{proof}

It is now ready to provide the error bounds on online optimistic PI (\ref{23}).
\begin{theorem}\label{t4}
Under Assumptions \ref{a1} and \ref{a2}, let $c\geq 0$ such that $J_0\geq T_0 J_0-c\nu$. Then for all $k\in\mathbb{N}$
\begin{align}
-\frac{\lambda_k(c)\nu}{1-\alpha_k}\leq J_k-J_k^*&\leq \alpha_0^k\|J_0-J_0^*\|\nu+\sum_{l=0}^{k-1}(J_l^*-J_{l+1}^*)                     \nonumber\\
&\hspace{0.4cm}+\sum_{l=1}^{k-1}(T_l^{k-l}-T_{l-1}^{k-l})J_l+e_k',               \label{26}
\end{align}
where $e_k':=\sum_{l=0}^{k-1}\frac{\alpha_l^{k-l}\lambda_l(c)\nu}{1-\alpha_l}$.
\end{theorem}
\begin{proof}
The proof is motivated by Lemma 2.5.4 in \cite{bertsekas2018abstract}. In light of $J_0\geq T_0J_0-c\nu$ and Lemma \ref{l1}, it can be obtained that
\begin{align}
J_k\geq T_k J_k-\lambda_k(c)\nu,~~\forall k\in\mathbb{N}         \nonumber
\end{align}
which in conjunction with Lemma 2.5.1(b) in \cite{bertsekas2018abstract} with $W=T_k,J=J_k$ and $k=0$ follows that
\begin{align}
J_k\geq J_k^*-\frac{\lambda_k(c)\nu}{1-\alpha_k},           \nonumber
\end{align}
thus ending the proof of (\ref{26}) on the left-hand side.

Now, invoking Lemma \ref{t1}, one has that
\begin{align}
T_jJ_j\geq T_{j+1}-\frac{\alpha_j\lambda_j(c)\nu}{1-\alpha_j},~~j=0,1,\ldots,k-1         \nonumber
\end{align}
which, together with Proposition 2.1.3 in \cite{bertsekas2018abstract} with $T_\mu=T_j^{k-j-1}$, leads to that
\begin{align}
T_{j}^{k-j}J_j\geq T_j^{k-j-1}J_{j+1}-\frac{\alpha_j^{k-j}\lambda_j(c)\nu}{1-\alpha_j}.        \nonumber
\end{align}
Summing the above inequality over $j=0,1,\ldots,k-1$ gives rise to that
\begin{align}
T_0^kJ_0\geq J_k+\sum_{l=1}^{k-1}(T_{l-1}^{k-l}-T_l^{k-l})J_l-e_k',            \nonumber
\end{align}
which implies that
\begin{align}
&T_0^kJ_0-J_0^*+\sum_{l=0}^{k-1}(J_l^*-J_{l+1}^*)+J_k^*            \nonumber\\
&\hspace{0.4cm}\geq J_k+\sum_{l=1}^{k-1}(T_{l-1}^{k-l}-T_l^{k-l})J_l-e_k'.       \nonumber
\end{align}
Using $\|T_0^kJ_0-J_0^*\|\leq \alpha_0^k\|J_0-J_0^*\|$ in the above inequality can obtain the right-hand side inequality in (\ref{26}). This completes the proof.
\end{proof}

\begin{remark}
It can be observed in the right-hand side of (\ref{26}) that the differences among $H_k$ will incur a larger error bound than (approximate) online VI and PI due to the presence of $\sum_{l=0}^{k-1}(J_l^*-J_{l+1}^*)$ and $\sum_{l=1}^{k-1}(T_l^{k-l}-T_{l-1}^{k-l})J_l$, resulting in accumulative errors as $k$ advances. However, due to $\lambda_k(c)\leq\frac{\eta_1}{1-\alpha^{m_d}}+c\alpha^{\sum_{l=0}^{k-1}m_l}$, where $\eta_1:=\max_{k\in\mathbb{N}}\eta_{1,k}$ and $m_d:=\min_{k\in\mathbb{N}}m_k$, the online optimistic PI has a faster convergence rate than $\alpha^k$ from one side, i.e., the left-hand side of (\ref{26}), with rate $\alpha^{\sum_{l=0}^{k-1}m_l}$. This result is consistent with the case where $H_k$'s are time-invariant, see Section 2.5.1 in \cite{bertsekas2018abstract}.
\end{remark}

\subsection{Approximate Online Optimistic Policy Iteration}

In this subsection, it is desirable to consider the approximate algorithm for the online optimistic PI, where both operations in (\ref{23}) are approximate. To be specific, {\em approximate online optimistic PI} generates sequences $\mu_k$ and $J_k$ by
\begin{subequations}
\begin{align}
&\|T_{k,\mu_{k+1}}J_k-T_kJ_k\|\leq \epsilon_k,                                    \label{27a}\\
&\|J_k-T_{k-1,\mu_k}^{m_k}J_{k-1}\|\leq \delta_k,~~\forall k\in\mathbb{N}         \label{27b}
\end{align}\label{27}
\end{subequations}
where $\epsilon_k,\delta_k\geq 0$ are some constants. As previously done, it is of help to introduce the metric to measure how different two consecutive $H_k$ and $H_{k+1}$ are, that is, there are constants $\eta_{2,k},\eta_{3,k}\geq 0$ such that for all $k\in\mathbb{N}$ and any $J\in\mathcal{B}(X),\mu\in\mathcal{M}$, and for $j\in\{1,k+1\}$
\begin{align}
\|(T_{k,\mu}^j-T_{k+1,\mu}^j) J\|\leq \eta_{2,k},~~\|J_k^*-J_{k+1}^*\|\leq \eta_{3,k}.        \label{29}
\end{align}
For example, in Example \ref{e1}, the first inequality in (\ref{29}) when $j=1$ means $\|g_k(x,\mu)-g_{k+1}(x,\mu)\|\leq \eta_{2,k}$ for all $x\in X,\mu\in U(x)$.

It is known from the case where $H_k$'s are time-invariant \cite{bertsekas2018abstract} that a stronger assumption than Assumptions \ref{a1} and \ref{a2} is required, and thus it is also employed here for the online case.
\begin{assumption}[Semilinear Monotonic Contraction]\label{a3}
For all $k\in\mathbb{N}$, $J\in\mathcal{B}(X)$ and $\mu\in\mathcal{M}$, there holds $T_{k,\mu}J,T_{k}J\in \mathcal{B}(X)$. Moreover, there exists $\alpha_k\in(0,1)$ for each $k\in\mathbb{N}$ such that for all $J_1,J_2\in\mathcal{B}(X),\mu\in\mathcal{M}$
\begin{align}
M(T_{k,\mu}J_1-T_{k,\mu}J_2)\leq \alpha_k M(J_1-J_2),          \label{30}
\end{align}
where the mapping $M:\mathcal{B}(X)\to\mathbb{R}$ is defined as $M(y)=\sup_{x\in X}\frac{y(x)}{\nu(x)}$ for a function $y\in\mathcal{B}(X)$.
\end{assumption}

With the above at hand, we are now in a position to give the error bound for approximate online optimistic PI.
\begin{theorem}\label{t5}
Under Assumption \ref{a3}, the sequences $\{\mu_k\}$ generated by (\ref{27}) satisfy
\begin{align}
\|J_{\mu_k}-J_k^*\|&\leq \frac{\alpha^{\sum_{l=1}^km_l}}{1-\alpha}M(T_{1,\mu_1}J_0-J_0)           \nonumber\\
&\hspace{0.0cm}+\alpha^{k-1}M(T_{1,\mu_1}J_0-J_1^*)+\frac{c_1\beta^{\lceil\frac{k}{2}\rceil}}{1-\beta}+\frac{c_1\beta\alpha^{\lfloor\frac{k}{2}\rfloor}}{1-\alpha}        \nonumber\\
&\hspace{0.0cm}+\frac{\alpha^{m_k}\varepsilon_1}{(1-\alpha)(1-\alpha^{m_d})}+\frac{\varepsilon_2}{1-\alpha},           \label{31}
\end{align}
where $c_1:=\frac{\alpha-\alpha^{m_s}}{1-\alpha}M(T_{1,\mu_1}J_0-J_0)$, $m_s:=\max_{k\in\mathbb{N}}m_k$, $\beta:=\alpha^{m_d}$, $m_d:=\min_{k\in\mathbb{N}}m_k$, $\varepsilon_1:=\epsilon+(1+\alpha)\delta+(2+\alpha)\eta_2$, $\varepsilon_2:=\frac{(\alpha-\alpha^{m_s})\varepsilon_1}{(1-\alpha)(1-\alpha^{m_d})}+\epsilon+\eta_2+\eta_3+\alpha(\delta+\eta_2)$, $\epsilon:=\max_{k\in\mathbb{N}}\epsilon_k$, $\delta:=\max_{k\in\mathbb{N}}\delta_k$, $\alpha:=\max_{k\in\mathbb{N}}\alpha_k$, $\eta_l:=\max_{k\in\mathbb{N}}\eta_{l,k}$ for $l=2,3$, and $\lfloor d\rfloor,\lceil d\rceil$ mean the largest integer not greater than $d$ and smallest integer not less than $d$ for a real number $d$, respectively.
\end{theorem}
\begin{proof}
This proof is adapted from Proposition 2.5.3 in \cite{bertsekas2018abstract}, which is given in the Appendix for the completeness.
\end{proof}

\begin{remark}
From (\ref{29}), it can be easily verified that the error bound on $\|J_{k,\mu_k}-J_k^*\|$ in the asymptotic sense is given as
\begin{align*}
\mathop{\lim\sup}_{k\to\infty}\|J_{k,\mu_k}-J_k^*\|\leq \frac{\hat{\alpha}\varepsilon_1}{(1-\alpha)(1-\alpha^{m_d})}+\frac{\varepsilon_2}{1-\alpha},
\end{align*}
where $\hat{\alpha}:=\alpha^{\mathop{\lim\inf}_{k\to\infty}m_k}$.
\end{remark}

\section{Asynchronous Algorithms}\label{s5}

This section aims at further alleviating the computational complexity by taking into account asynchronous algorithms.

\subsection{Asynchronous Approximate Online Value Iteration}

Consider that there are $N$ processors for solving online abstract DP, and partition the state set $X$ into $N$ disjoint nonempty subsets $X_1,\ldots,X_N$. Correspondingly, let us partition $J$ as $J=(J_1,\ldots,J_N)$, where $J_l$ is the restriction of $J$ on $X_l$ for $l\in[N]$ with the notation $[N]:=\{1,\ldots,N\}$. Let $\mathcal{T}_l$ be a subset of iterations, denoting the updating or activation of processor $l\in[N]$. Then the {\em asynchronous approximate online VI} is given as
{\small\begin{align}
J_{l,k+1}(x)=\left\{
               \begin{array}{ll}
                 \tilde{T}_k^{m_k}(J_{1,\tau_{l1}(k)},\cdots,J_{N,\tau_{lN}(k)})(x), & k\in\mathcal{T}_l,x\in X_l \\
                 J_{l,k}(x), & k\notin\mathcal{T}_l,x\in X_l
               \end{array}
             \right.              \label{32}
\end{align}}
where $\tau_{li}(k)\in\{0,1,\ldots,k\}$ with $k-\tau_{li}(k)$ being the communication delay from processor $i\in[N]$ to processor $l$.

In the online case, some conditions on updating frequency and communication delays are listed below.
\begin{assumption}[Continuous Updating and Uniformly Bounded Delay]\label{a4}
~~~
\begin{enumerate}
  \item There exists an integer $T_a>0$ such that $\mathcal{T}_l\cap[k,k+T_a]\neq\emptyset$ for all $k\in\mathbb{N}$ and $l\in[N]$;
  \item There holds $|k-\tau_{ij}(k)|\leq T_d$ for some integer $T_d\geq 0$, for all $k\in\mathbb{N}$ and $i,j\in[N]$.
\end{enumerate}
\end{assumption}

The first condition in the above assumption means that each processor must update or activate at least once within consecutive $T_a$ time instants, and the second one indicates an upper bound on the communication delays.

With the above preparations, it is ready to develop the error bound on asynchronous approximate online VI.
\begin{theorem}\label{t6}
Under conditions (\ref{15})-(\ref{16}), Assumption \ref{a2} with $\nu(x)\equiv 1$ and Assumption \ref{a4}, the sequence $\{J_{k}\}$ generated by (\ref{32}) satisfies
\begin{align}
\limsup_{k\to\infty}\|J_k-J_k^*\|\leq \frac{\rho(T_a+\alpha^{m_d}T_d)+e}{1-\alpha^{m_d}},          \label{33}
\end{align}
where $\rho$ and $e$ are defined after (\ref{16}) and $m_d=\min_{k\in\mathbb{N}}m_k$.
\end{theorem}
\begin{proof}
Consider the time step $k+1$ and processor $l$. To simplify the notations, denote by $J_{\tau_l(k)}:=(J_{1,\tau_{l1}(k)},\cdots,J_{N,\tau_{lN}(k)})$. The analysis is divided into two cases: $k\in\mathcal{T}_l$ and $k\notin\mathcal{T}_l$.

If $k\in\mathcal{T}_l$, then one has
\begin{align}
&|J_{l,k+1}(x)-J_{k+1}^*(x)|           \nonumber\\
&=|\tilde{T}_k^{m_k}(J_{\tau_l(k)})(x)-J_{k+1}^*(x)|          \nonumber\\
&\leq |T_k^{m_k}(J_{\tau_l(k)})(x)-J_k^*(x)|+|J_k^*(x)-J_{k+1}^*(x)|             \nonumber\\
&\hspace{0.4cm}+|\tilde{T}_k^{m_k}(J_{\tau_l(k)})(x)-T_k^{m_k}(J_{\tau_l(k)})(x)|         \nonumber\\
&\leq \alpha^{m_k}\|J_{\tau_l(k)}-J_k^*\|+\rho+e                                     \nonumber\\
&\leq \alpha^{m_k}\big(\|J_{\tau_l(k)}-J_{\tau_l(k)}^*\|+\cdots+\|J_{k-1}^*-J_k^*\|\big)+\rho+e        \nonumber\\
&\leq \alpha^{m_k}\|J_{\tau_l(k)}-J_{\tau_l(k)}^*\|+\alpha^{m_k}T_d\rho+\rho+e,                \nonumber
\end{align}
where the second condition in Assumption \ref{a4} has been exploited to obtain the last inequality.

If $k\notin \mathcal{T}_l$, then there must exist an integer $t'\in[k+1-T_a,k)$ such that processor $l$ updates or activates at time slot $t'$. As a result, one can obtain that
\begin{align}
&|J_{l,k+1}(x)-J_{k+1}^*(x)|         \nonumber\\
&=|J_{l,k}(x)-J_{k+1}^*(x)|=\cdots=|J_{l,t'+1}(x)-J_{k+1}^*(x)|         \nonumber\\
&=|\tilde{T}_{t'}^{m_{t'}}J_{\tau_l(t')}(x)-J_{k+1}^*(x)|              \nonumber\\
&\leq |T_{t'}^{m_{t'}}J_{\tau_l(t')}(x)-J_{t'}^*(x)|+|J_{t'}^*(x)-J_{k+1}^*(x)|         \nonumber\\
&\hspace{0.4cm}+|\tilde{T}_{t'}^{m_{t'}}J_{\tau_l(t')}(x)-T_{t'}^{m_{t'}}J_{\tau_l(t')}(x)|         \nonumber\\
&\leq \alpha^{m_{t'}}\|J_{\tau_l(t')}-J_{t'}^*\|+\sum_{i=t'}^k|J_i^*(x)-J_{i+1}^*(x)|+e            \nonumber\\
&\leq \alpha^{m_{t'}}\|J_{\tau_l(t')}-J_{\tau_l(t')}^*\|+\alpha^{m_{t'}}T_d\rho+T_a\rho+e,            \nonumber
\end{align}
where the similar technique to the last step of the above inequality has been used to obtain the last inequality.

Combining the above two inequalities yields that
\begin{align}
\|J_k-J_k^*\|&\leq \alpha^{m_d}\|J_{k-\tau(k)}-J_{k-\tau(k)}^*\|       \nonumber\\
&\hspace{0.4cm}+\rho(T_a+\alpha^{m_d}T_d)+e,         \nonumber
\end{align}
where $\tau(k)\in\{1,\ldots,T_a+T_d\}$. Consequently, in view of Lemma \ref{l0}, the conclusion can be obtained.
\end{proof}

\begin{remark}
Note that in the case where $H_k$'s are time-invariant \cite{bertsekas2018abstract}, the asynchronous value iteration is anatomized under less conservative conditions than Assumption \ref{a4}, i.e., each set $\mathcal{T}_l$ is infinite for all $l\in[N]$ and $\lim_{k\to\infty}\tau_{li}(k)=\infty$ for all $l,i\in[N]$. However, the analysis under the aforementioned conditions is no longer available to the online case studied in this paper.
\end{remark}

\subsection{Asynchronous Online Policy Iteration}

This subsection is to study the asynchronous algorithms for online policy iteration. To do so, let us first review the case of $H_k$'s being time-invariant. It is known that the natural asynchronous version of optimistic PI is not reliable in general, having a possibility of oscillation, and thus two another asynchronous PI algorithms have been proposed in \cite{bertsekas2018abstract}, i.e., an optimistic asynchronous algorithm with randomization and a policy iteration with a uniform fixed point. Usually, the first algorithm has some restrictions, for instance, assuming totally finite policies. In contrast, the second one is more advantageous without such restriction. Hence, the second algorithm is only take into consideration for the online case in this subsection. The idea is to introduce new functions to eliminate the anomaly that $T_k$ and $T_{k,\mu}$ do not have identical fixed points.

To do so, it is necessary to introduce two additional functions
\begin{align}
V:X\to\mathbb{R}~~\text{and}~~Q:X\times U\to\mathbb{R},           \label{34}
\end{align}
referred to as a cost function and $Q$-factor as in the DP context, respectively. Meanwhile, for all $k\in\mathbb{N}$, define two functions $F_{k,\mu}(V,Q)$ and $MF_{k,\mu}(V,Q)$ as
\begin{align}
F_{k,\mu}(V,Q)(x,u)&:=H_k(x,u,\min\{V,Q_\mu\}),       \label{35}\\
MF_{k,\mu}(V,Q)(x)&:=\min_{u\in U(x)}F_{k,\mu}(V,Q)(x,u),        \label{36}
\end{align}
where $Q_\mu(x):=Q(x,\mu(x))$ for all $x\in X$.

Now, a new mapping $G_{k,\mu}$ is defined as
\begin{align}
G_{k,\mu}(V,Q):=(MF_{k,\mu}(V,Q),F_{k,\mu}(V,Q)).     \label{37}
\end{align}
and the norm is defined by
\begin{align}
\|(V,Q)\|:=\max\{\|V\|,\|Q\|\},                   \label{38}
\end{align}
where $\|V\|$ is the weighted sup-norm of $V$, and
\begin{align}
\|Q\|:=\sup_{x\in X,u\in U(x)}\frac{|Q(x,u)|}{\nu(x)}.            \label{39}
\end{align}
Some good properties have been shown for $G_{k,\mu}$ in Proposition 2.6.4 in \cite{bertsekas2018abstract}, that is, for each fixed $k\in\mathbb{N}$ under Assumption \ref{a2}, $G_{k,\mu}$ has a unique fixed point $(J_k^*,Q_k^*)$, in which $Q_k^*$ is defined as $Q_k^*(x,u)=H_k(x,u, J_k^*)$ for $x\in X,u\in U(x)$, and $G_{k,\mu}$ is contractive in the sense
\begin{align}
&\|G_{k,\mu}(V_1,Q_1)-G_{k,\mu}(V_2,Q_2)\|        \nonumber\\
&\hspace{3.0cm}\leq \alpha_k\|(V_1,Q_1)-(V_2,Q_2)\|.         \label{40}
\end{align}

As in the last subsection, let us consider $N$ processors and divide the set $X$ into $N$ parts as $X_1,\ldots,X_N$, each of which is assigned to a separate processor. Each processor $l\in[N]$ maintains $V_k(x)$, $Q_k(x,u)$, and $\mu_k(x)$ only for $x$ in its local set $X_l$, and enjoy disjoint activation or updating time set $\mathcal{T}_l$ and $\bar{\mathcal{T}}_l$ for all processors $l\in[N]$.

At this position, the {\em asynchronous online PI} is proposed as
\begin{enumerate}
  \item {\em Online local policy improvement:} If $k\in\mathcal{T}_l$, processor $l$ updates that for all $x\in X_l$
\begin{align}
V_{k+1}(x)&=MF_{k,\mu_k}(V_k,Q_k)(x),          \nonumber\\
\mu_{k+1}(x)&=\mathop{\arg\min}_{u\in U(x)}H_k(x,u,\min\{V_k,Q_{k,\mu_k}\}),                                    \label{41}
\end{align}
and $Q_{k+1}(x,u)=Q_k(x,u)$ for all $x\in X_l,u\in U(x)$.
  \item {\em Online local policy evaluation:} If $k\in\bar{\mathcal{T}}_l$, processor $l$ updates for all $x\in X_l$ and $u\in U(x)$
\begin{align}
Q_{k+1}(x,u)=F_{k,\mu_k}(V_k,Q_k)(x,u),         \label{42}
\end{align}
and $V_{k+1}(x)=V_k(x)$, $\mu_{k+1}(x)=\mu_k(x)$ for all $x\in X_l$.
\end{enumerate}

To proceed, the following assumptions are of help for the subsequent analysis.

\begin{assumption}[Bounds on Consecutive Optimal Costs and Updating Frequency]\label{a5}
~~~
\begin{enumerate}
  \item There exists a constant $\bar{\rho}_k$ such that $\|(J_k^*,Q_k^*)-(J_{k+1}^*,Q_{k+1}^*)\|\leq \bar{\rho}_k$;
  \item There exists an integer $T_a>0$ such that $\mathcal{T}_l\cap [k,k+T_a]\neq \emptyset$ and $\bar{\mathcal{T}}_l\cap[k,k+T_a]\neq \emptyset$ for all $k\in\mathbb{N}$ and  $l\in[N]$.
\end{enumerate}
\end{assumption}

At present, it is ready to establish the following error bound result.
\begin{theorem}\label{t7}
Under Assumptions \ref{a2} and \ref{a5}, for the sequence $\{(V_k,Q_k)\}$ generated by asynchronous online PI (\ref{41})-(\ref{42}), there holds that
\begin{align}
\limsup_{k\to\infty}\|(V_k,Q_k)-(J_k^*,Q_k^*)\|\leq \frac{\bar{\rho}T_a}{1-\alpha},          \label{43}
\end{align}
where $\bar{\rho}:=\max_{k\in\mathbb{N}}\bar{\rho}_k$ and $\alpha:=\max_{k\in\mathbb{N}}\alpha_k$.
\end{theorem}
\begin{proof}
For any $k>0$, based on Assumption \ref{a5}(1), there must exist two constants $t_1,t_2\in [k+1-T_a,k]$ such that processor $l\in[N]$ performs online local policy improvement and evaluation, respectively. Therefore, it can be concluded that for all $x\in X,u\in U(x)$
\begin{align}
|V_{k+1}(x)-&J_{k+1}^*(x)|=|MF_{t_1,\mu_{t_1}}(V_{t_1},Q_{t_1})(x)-J_{k+1}^*(x)|            \nonumber\\
&\leq |MF_{t_1,\mu_{t_1}}(V_{t_1},Q_{t_1})(x)-J_{t_1}^*(x)|            \nonumber\\
&\hspace{0.4cm}+\sum_{i=t_1}^k|J_i^*(x)-J_{i+1}^*(x)|                       \nonumber\\
&\leq |MF_{t_1,\mu_{t_1}}(V_{t_1},Q_{t_1})(x)-J_{t_1}^*(x)|+\bar{\rho}T_a,              \nonumber
\end{align}
and
\begin{align}
&|Q_{k+1}(x,u)-Q_{k+1}^*(x,u)|          \nonumber\\
&=|F_{t_2,\mu_{t_2}}(V_{t_2},Q_{t_2})(x,u)-Q_{k+1}^*(x,u)|          \nonumber\\
&\leq |F_{t_2,\mu_{t_2}}(V_{t_2},Q_{t_2})(x,u)-Q_{t_2}^*(x,u)|       \nonumber\\
&\hspace{0.4cm}+\sum_{i=t_2}^k|Q_i^*(x,u)-Q_{i+1}^*(x,u)|        \nonumber\\
&\leq |F_{t_2,\mu_{t_2}}(V_{t_2},Q_{t_2})(x,u)-Q_{t_2}^*(x,u)|+\bar{\rho}T_a,
\end{align}
where Assumption \ref{a5} has been applied to obtain the last inequalities of the above two expressions.

As a consequence, it can be obtained that
\begin{align}
&\|(V_{k+1},Q_{k+1})-(J_{k+1}^*,Q_{k+1}^*)\|          \nonumber\\
&=\max\{\|V_{k+1}-V_{k+1}^*\|,\|Q_{k+1}-Q_{k+1}^*\|\}       \nonumber\\
&\leq \max\{\|MF_{t_1,\mu_{t_1}}(V_{t_1},Q_{t_1})-J_{t_1}^*\|,       \nonumber\\
&\hspace{1.3cm}\|F_{t_2,\mu_{t_2}}(V_{t_2},Q_{t_2})-Q_{t_2}^*\|\}+\bar{\rho}T_a       \nonumber\\
&\leq \max\{\|G_{t_1,\mu_{t_1}}(V_{t_1},Q_{t_1})-(J_{t_1}^*,Q_{t_1}^*)\|,       \nonumber\\
&\hspace{1.3cm}\|G_{t_2,\mu_{t_2}}(V_{t_2},Q_{t_2})-(J_{t_2}^*,Q_{t_2}^*)\|\}+\bar{\rho}T_a,       \nonumber
\end{align}
which implies that there must exist a constant $\tau\in[k-T_a,k-1]$ such that
\begin{align}
&\|(V_{k},Q_{k})-(J_{k}^*,Q_{k}^*)\|          \nonumber\\
&\leq \|G_{\tau,\mu_{\tau}}(V_{\tau},Q_{\tau})-(J_{\tau}^*,Q_{\tau}^*)\|+\bar{\rho}T_a       \nonumber\\
&\leq \alpha\|(V_\tau,Q_\tau)-(J_{\tau}^*,Q_\tau^*)\|+\bar{\rho}T_a.                  \nonumber
\end{align}
Invoking Lemma \ref{l0} to the above inequality gives rise to the desired conclusion (\ref{43}), which completes the proof.
\end{proof}

\begin{remark}
It should be pointed out that communication delays, approximate algorithms, and multiple iterations at single step can be similarly addressed for the asynchronous online PI as previously done in this paper.
\end{remark}

It can be observed that it is not necessary to evaluate $Q$ over the entire state space (its value at $\mu_k(x)$ is enough), since the goal is only to calculate $J_k^*$. Consequently, by letting $J_k(x):=Q_k(x,\mu_k(x))$ for all $x\in X$, iterations (\ref{41}) and (\ref{42}) in the asynchronous online PI can, respectively, reduce to
\begin{align}
J_{k+1}=V_{k+1}(x)&=\min_{u\in U(x)}H_k(x,u,\min\{V_k,J_{k}\}),              \nonumber\\
\mu_{k+1}(x)&=\mathop{\arg\min}_{u\in U(x)}H_k(x,u,\min\{V_k,J_{k}\}),        \label{44}\\
J_{k+1}(x,u)&=H_k(x,u,\min\{V_k,J_k\})(x,u).                              \label{45}
\end{align}

\section{Examples}\label{s6}

In Example \ref{e1}, an online optimal control problem has been introduced to illustrate the problem formulation for online abstract DP, where $H_k$ is defined by $H_k(x,u,J)=g_k(x,u)+\alpha J(f(x,u))$. It is straightforward to see that $H_k$ satisfies Assumption \ref{a1}, and given $\alpha\in (0,1)$ and the boundedness of $g_k$, Assumption \ref{a2} is also satisfied by $H_k$ with respect to standard unweighted sup-norm, i.e., $\nu\equiv 1$. As a result, the theoretical results in this paper can be applied to the problem in Example \ref{e1}.

\begin{example}[Online Finite-State Discounted MDPs]\label{e2}
As another example, consider online finite-state discounted MDPs, which involves a system $x_{k+1}=f(x_k,u_k,w_k),k\in\mathbb{N}$ with finite states, where $x_k\in X$ is the state, $u_k\in U$ is the control, and $w_k\in W$ is a random disturbance with $W$ being countable. Also, the state equation is given in terms of transition probabilities
\begin{align}
p_{xy}(u)=Prob(y=f(x,u,w)|x),          \label{46}
\end{align}
for all $x,y\in X$ and $u\in U(x)$. In the meantime, taking into account a cost function $g_k(x,u)$ at each time step $k\in\mathbb{N}$. Then the abstract DP mapping $H_k$ can be written as
\begin{align}
H_k(x,u,J)=\sum_{y\in X} p_{xy}(u)(g_k(x,u,y)+\alpha J(y)).          \label{47}
\end{align}
It is easy to verify that $H_k$ is monotone, thus satisfying Assumption \ref{a1}. Moreover, if $\alpha\in(0,1)$ and $g_k$ are bounded, then $H_k$ is also contractive with respective to the standard unweighted sup-norm.

As a consequence, the online algorithms in previous sections are applicable to this problem. For instance, asynchronous online PI can be leveraged in which case the function in (\ref{35}) can be explicitly written as
\begin{align}
F_{k,\mu}(V,Q)(x,u)&=\sum_{y\in X}p_{xy}(u)\Big(g_{k}(x,u,y)           \nonumber\\
&\hspace{0.4cm}+\alpha\min\{V(y),Q(y,\mu(y))\}\Big).           \label{48}
\end{align}
\end{example}

Basically, all those problems, which satisfy monotone and contractive assumptions in the stationary case, i.e., $H_k$'s being independent of time, will still meet the two assumptions in the online case.

\section{Conclusion}\label{s7}

This paper has studied the online abstract DP problems, where the abstract mappings are time-varying, leading to that the optimal costs and policies are time-varying as well. It is known that to accurately track time-varying optimal costs and polices is in general impossible in the online case, thus necessitating the investigation on this problem. In this paper, we have developed quite a few algorithms based on classical ones in the static case where $H_k$'s are independent of time, and the tracking error bounds have been provided for these online algorithms, including approximate online VI, online PI, approximate online PI, online optimistic PI, approximate online optimistic PI, and asynchronous online PI and VI algorithms. It has been shown that the largest difference between consecutive abstract mappings $H_k$ and $H_{k+1}$ for $k\in\mathbb{N}$ play a critical part in the tracking error bounds. This paper focuses on the contractive models, as a first step to investigate the abstract DP in the online case, and thereby the future directions can be placed on the online abstract DP with semicontractive and noncontractive models.

\section*{Acknowledgment}

The authors would like to thank Dr. Zhirong Qiu for his helpful suggestions on this paper.


\section*{Appendix}

\noindent{\em The Proof of Theorem \ref{t5}:}

Throughout this proof, for notation ease, let $T_{\mu_k}$ (resp. $J_{\mu_k}$) simply denote $T_{k,\mu_k}$ (resp. $J_{k,\mu_k}$) when having the same time $k$, where $J_{k,\mu}$ means the fixed point of $T_{k,\mu}$, and denote
\begin{align}
&\underline{J}=J_{k-1}, J=J_k, \mu=\mu_k, \overline{\mu}=\mu_{k+1}, m=m_k, \overline{m}=m_{k+1},         \nonumber\\
&J^*=J_k^*, \overline{J}^*=J_{k+1}^*, s=J_\mu-T_\mu^m\underline{J}, \overline{s}=J_{\overline{\mu}}-T_{\overline{\mu}}^{\overline{m}}J,      \nonumber\\
&t=T_\mu^m\underline{J}-J^*, \overline{t}=T_{\overline{\mu}}^{\overline{m}}J-\overline{J}^*, r=T_\mu\underline{J}-\underline{J}, \overline{r}=T_{\overline{\mu}}J-J.     \nonumber
\end{align}
Then, it is easy to see that
\begin{align}
J_\mu-J^*=s+t.       \nonumber
\end{align}
In what follows, let us develop the bounds on $M(r)$, $M(s)$, and $M(t)$.

First, consider $M(r)$. It can be obtained that
{\small\begin{align*}
\overline{r}&=T_\mu J-J=(T_{\overline{\mu}}J-T_\mu J)+(T_\mu J-J)          \\
&\leq (T_{\overline{\mu}}J-T_k J)+(T_\mu J-T_\mu(T_\mu^m \underline{J}))      \\
&\hspace{0.4cm}+(T_\mu^m\underline{J}-J)+(T_\mu^m(T_\mu\underline{J})-T_\mu^m\underline{J})        \\
&\leq (T_{k,\overline{\mu}}J-T_kJ)+(T_{\overline{\mu}} J-T_{k,\overline{\mu}}J)+\alpha M(J-T_\mu^m\underline{J})\nu         \\
&\hspace{0.4cm}+(T_{k-1,\mu}^m\underline{J}-J)+(T_\mu^m\underline{J}-T_{k-1,\mu}^m\underline{J})+\alpha^m M(T_\mu\underline{J}-\underline{J})\nu      \\
&\leq (\epsilon+\delta)\nu+2\eta_2\nu+\alpha^m M(r)\nu+\alpha M(J-T_{k-1,\mu}^m\underline{J})\nu           \\
&\hspace{0.4cm}+\alpha M(T_{k-1,\mu}^m\underline{J}-T_{\mu}^m\underline{J})\nu         \\
&\leq (\epsilon+\delta)\nu+(2\eta_2+\alpha\delta+\alpha\eta_2)\nu+\alpha^m M(r)\nu,
\end{align*}}
where (\ref{29}) has been utilized to obtain the last two inequalities, which implies that
\begin{align*}
M(\overline{r})\leq \alpha^m M(r)+\varepsilon_1.
\end{align*}
By defining $M_{r,k}:=M(r)$, one has $M_{r,k+1}=M(\overline{r})$, and thus, by recursively iterating the above inequality, it yields that
\begin{align}
M_{r,k}&\leq \alpha^{\sum_{l=1}^{k-1}m_l}M_{r,1}+\varepsilon_1\sum_{j=1}^{k-1}\alpha^{\sum_{l=j+1}^{k-1}m_l}           \nonumber\\
&\leq \alpha^{\sum_{l=1}^{k-1}m_l}M_{r,1}+\frac{\varepsilon_1}{1-\alpha^{m_d}},               \label{pf2}
\end{align}
with the convention $\alpha^{\sum_{l=k}^{k-1}m_l}=1$.

Now, consider the bound on $M(s)$. To do so, invoking Proposition 2.1.4(b) in \cite{bertsekas2018abstract} gives rise to
\begin{align*}
J_\mu\leq \underline{J}+\frac{T_\mu\underline{J}-\underline{J}}{1-\alpha_k}\leq \underline{J}+\frac{T_\mu\underline{J}-\underline{J}}{1-\alpha},
\end{align*}
which together Assumption \ref{a3} follows that
\begin{align*}
s&=J_\mu-T_\mu^m\underline{J}=T_\mu^m J_\mu-T_\mu^m\underline{J}\leq \alpha^m M(J_\mu-\underline{J})\nu          \\
&\leq \frac{\alpha^m}{1-\alpha}M(T_\mu\underline{J}-\underline{J})\nu            \\
&\leq \frac{\alpha^m}{1-\alpha}M(r)\nu,
\end{align*}
further implying that
\begin{align}
M(s)&\leq \frac{\alpha^m}{1-\alpha}M(r)            \nonumber\\
&\leq \frac{\alpha^{\sum_{l=1}^k m_l}}{1-\alpha}M_{r,1}+\frac{\varepsilon_1\alpha^m}{(1-\alpha)(1-\alpha^{m_d})},        \label{pf3}
\end{align}
where (\ref{pf2}) has been used in the last inequality.

In what follows, let us focus on the bound on $M(t)$. Some manipulations with (\ref{29}) lead to that
\begin{align*}
\overline{t}&=T_{\overline{\mu}}^{\overline{m}}J-J^*+J^*-\overline{J}^*           \\
&=(T_{\overline{\mu}}^{\overline{m}}J-T_{\overline{\mu}}^{\overline{m}-1}J)+\cdots+(T_{\overline{\mu}}^{2}J-T_{\overline{\mu}}J)               \\
&\hspace{0.4cm}+(T_{\overline{\mu}}J-T_k J)+(T_k J-T_k J^*)+(J^*-\overline{J}^*)              \\
&\leq (\alpha^{\overline{m}-1}+\cdots+\alpha)M(T_{\overline{\mu}}J-J)\nu+(T_{\overline{\mu}}J-T_{k,\overline{\mu}}J)        \\
&\hspace{0.4cm}+(T_{k,\overline{\mu}}J-T_k J)+(T_k J-T_k J^*)+(J^*-\overline{J}^*)          \\
&\leq \frac{\alpha-\alpha^{\overline{m}}}{1-\alpha}M(\overline{r})\nu+(\epsilon+\eta_2+\eta_3)\nu+(T_kJ-T_kJ^*).
\end{align*}

Take into account the term $T_kJ-T_kJ^*$ in the last inequality. In view of Assumption \ref{a3} and (\ref{29}), one can obtain that
\begin{align*}
T_kJ-T_kJ^*&\leq \alpha M(J-J^*)\nu             \\
&\leq \alpha [M(J-T_{k-1,\mu}^m \underline{J})+M(T_{k-1,\mu}^m \underline{J}-T_{\mu}^m \underline{J})        \\
&\hspace{0.4cm}+M(T_{\mu}^m \underline{J}-J^*)]\nu        \\
&\leq \alpha(\delta+\eta_2)\nu+\alpha M(t)\nu,
\end{align*}
which in conjunction with the above inequality results in that
\begin{align*}
\overline{t}&\leq \frac{\alpha-\alpha^{\overline{m}}}{1-\alpha}M(\overline{r})\nu+(\epsilon+\eta_2+\eta_3)\nu     \\
&\hspace{0.4cm}+\alpha(\delta+\eta_2)\nu+\alpha M(t)\nu.
\end{align*}
Hence, in light of (\ref{pf2}), it can be concluded that
\begin{align*}
M(\overline{t})&\leq \frac{\alpha-\alpha^{\overline{m}}}{1-\alpha}M(\overline{r})+\epsilon+\eta_2+\eta_3+\alpha(\delta+\eta_2)+\alpha M(t)         \\
&\leq c_1\beta^k+\varepsilon_2+\alpha M(t),
\end{align*}
which, after defining $M_{t,k}:=M(t)$, follows that
\begin{align*}
M_{t,k+1}\leq \alpha M_{t,k}+c_1\beta^k+\varepsilon_2.
\end{align*}
As a result, it is straightforward to verify that
\begin{align}
M_{t,k}\leq \alpha^{k-1}M_{t,1}+c_1\sum_{l=0}^{k-2}\alpha^l\beta^{k-l-1}+\frac{\varepsilon_2}{1-\alpha}.           \label{pf4}
\end{align}

Equipped with the above preparations, making use of (\ref{pf3})-(\ref{pf4}), one has that
\begin{align*}
M(J_{\mu_k}-J_k^*)&\leq M(s)+M(t)           \\
&\leq \frac{\alpha^{\sum_{l=1}^k m_l}}{1-\alpha}M_{r,1}+\frac{\varepsilon_1\alpha^{m_k}}{(1-\alpha)(1-\alpha^{m_d})}            \\
&\hspace{0.4cm}+\alpha^{k-1}M_{t,1}+c_1\sum_{l=0}^{k-2}\alpha^l\beta^{k-l-1}+\frac{\varepsilon_2}{1-\alpha},
\end{align*}
which, in combination with the fact that $J_{\mu_k}\geq J_k^*$ by Proposition 2.1.2 in \cite{bertsekas2018abstract}, follows that
\begin{align*}
\|J_{\mu_k}-J_k^*\|&\leq \frac{\alpha^{\sum_{l=1}^km_l}}{1-\alpha}M_{r,1}+\alpha^{k-1}M_{t,1}+\frac{\varepsilon_2}{1-\alpha}           \nonumber\\
&\hspace{0.4cm}+\frac{\alpha^{m_k}\varepsilon_1}{(1-\alpha)(1-\alpha^{m_d})}+c_1\sum_{l=0}^{k-2}\alpha^l\beta^{k-l-1}.
\end{align*}

In the last inequality, the term $\sum_{l=0}^{k-2}\alpha^l\beta^{k-l-1}$ can be analyzed as
\begin{align*}
\sum_{l=0}^{k-2}\alpha^l\beta^{k-l-1}&=\big(\alpha^0\beta^{k-1}+\alpha\beta^{k-2}+\cdots+\alpha^{\lfloor\frac{k}{2}\rfloor-1}\beta^{k-\lfloor\frac{k}{2}\rfloor}\big)      \\
&\hspace{0.4cm}+\big(\alpha^{\lfloor\frac{k}{2}\rfloor}\beta^{k-1-\lfloor\frac{k}{2}\rfloor}+\cdots+\alpha^{k-2}\beta\big)           \\
&\leq \frac{\beta^{k-\lfloor\frac{k}{2}\rfloor}}{1-\beta}+\frac{\beta\alpha^{\lfloor\frac{k}{2}\rfloor}}{1-\alpha}         \\
&= \frac{\beta^{\lceil\frac{k}{2}\rceil}}{1-\beta}+\frac{\beta\alpha^{\lfloor\frac{k}{2}\rfloor}}{1-\alpha},
\end{align*}
where the last equality has employed the fact that $k=\lfloor\frac{k}{2}\rfloor+\lceil\frac{k}{2}\rceil$, by substituting which into the last inequality one can obtain the inequality (\ref{31}). This ends the proof.
\hfill\rule{2mm}{2mm}




\end{document}